\documentclass[twoside,12pt]{article}
\usepackage{amsmath, amsthm, amscd, amsfonts, amssymb, graphicx, color}
\usepackage{graphicx}
 \usepackage[vmargin=3cm,left=3cm,right=3cm]{geometry}
\begin{document}
\setcounter{page}{1}
\setlength{\unitlength}{12mm}
\newcommand{\f}{\frac}
\newtheorem{theorem}{Theorem}[section]
\newtheorem{lemma}[theorem]{Lemma}
\newtheorem{proposition}[theorem]{Proposition}
\newtheorem{corollary}[theorem]{Corollary}
\theoremstyle{definition}
\newtheorem{definition}[theorem]{Definition}
\newtheorem{example}[theorem]{Example}
\newtheorem{solution}[theorem]{Solution}
\newtheorem{xca}[theorem]{Exercise}
\theoremstyle{remark}
\newtheorem{remark}[theorem]{Remark}
\numberwithin{equation}{section}
\newcommand{\sta}{\stackrel}
%_________________________Pleae insert full name and address of authors______________________________________
\title{\bf Polar fuzzy topologies on fuzzy topological vector spaces}
\author { B. Daraby$^{a}$, N. Khosravi$^{b}$, A. Rahimi$^{c}$}
\date{\footnotesize $^{a, b, c}$Department of Mathematics, University of Maragheh, P. O. Box 55136-553, Maragheh, Iran\\}
\maketitle
%_____________________________________________________________________
%_________________________Please do not change this part any way____________________________________________

%_____________________________________________________________________
%_________________________Please type the abstract here_____________________________________________________
\begin{abstract}
\noindent
In this paper, we introduce the concept of polar fuzzy sets on fuzzy dual spaces. Using the notion of polar fuzzy sets, we define polar linear fuzzy topologies on fuzzy dual spaces and prove the Mackey-Arens type Theorem on fuzzy topological vector spaces.

%of this paper is showing the Pr\'{e}kopa-Leindler type inequality for Sugeno integral. Also, we discuss on $\mu$ conditions for proof of inequality.

\vspace{.5cm}\leftline{{Subject Classification 2010: 03E72, 54A40
}}

\vspace{.5cm}\leftline{Keywords: Polar fuzzy sets, Polar fuzzy topology, Mackey-Arens Theorem.}
\end{abstract}
%_____________________________________________________________________
\section{Introduction}
Katsaras and Liu defined the notion of fuzzy topological vector space in \cite{k10}. Later, Katsaras changed the definition of fuzzy topological vector space  in \cite{k7} and he considered the linear fuzzy topology on the scalar field $\mathbb{K}$, which consists of all lower semi-continuous functions from $\mathbb{K}$ into $I=[0,1]$. Also,  Katsaras \cite{k8} investigated the concepts of fuzzy norm and fuzzy seminorm on vector spaces. After then, many authors such as, Cheng and Mordeson \cite{f}, Felbin \cite{g}, Bag and Samanta \cite{poi}, and so on started to introduce the notion of fuzzy normed linear space. Later, Xiao and Zhu \cite{u}, Fang \cite{dD}, Daraby et. al. (\cite{d}, \cite{d2}) redefined the notion of Felbin's \cite{g} definition of fuzzy norm and studied some properties of its topological structure. Das \cite{kh} introduced a fuzzy topology generated by fuzzy norm and studied some properties of this topology. Fang \cite{dD} introduced a new I-topology $ \mathcal{T}^*_{\Vert.\Vert} $ on the fuzzy normed vector space using the notion of fuzzy norm.

The Mackey-Arens Theorem has an important role in the theory of convex spaces. For proving this theorem in fuzzy topological vector spaces, we need to extend the concept of polar sets to fuzzy topological vector spaces. In this paper, at the first, we define the polar fuzzy sets and then introduce the concept of polar fuzzy topology. Finally, we  prove the Mackey-Arens type Theorem in fuzzy topological vector spaces.

%%%%%%%%%%%%%%%%%%%%%%%%%

\section{Preliminaries}
Let $ X $ be a non-empty set. A fuzzy set in $X$ is a function from $X$ into $I$.
\begin{definition}\cite{kh} Let $X$ and $Y$ be any two non-empty sets, $ f: X \rightarrow Y $ be a mapping and $ \mu $ be a fuzzy subset of $X$. Then $f(\mu)$ is a fuzzy subset of $Y$ defined by
$$f(\mu)(y)=\begin{cases}\sup_{x\in f^{-1}(y)} \mu(x) &\quad f^{-1}(y)\neq\emptyset,\\0 & \:\:\:\: else,\end{cases}$$
for all $y \in Y$, where $f^{-1}(y)=\lbrace x: f(x)=y \rbrace$. If $\eta$ is a fuzzy subset of $Y$, then the fuzzy subset $f^{-1}(\eta)$ of $X$ is defined by $f^{-1}(\eta)(x)=\eta(f(x))$ for all $x\in X$.
\end{definition}
\begin{definition}\cite{k7}
A fuzzy topology on a set $ X $ is a subset $ \tau_f $ of $ I^{X} $ satisfying the following conditions:
\begin{enumerate}
\item[$(i)$] $ \tau_f $ contains every constant fuzzy set in $ X $,
\item[$(ii)$] if $ \mu_{1},\mu_{2} \in \tau_f $, then $ \mu_{1} \wedge \mu_{2} \in \tau_f $,
\item[$(iii)$] if $ \mu_{i} \in \tau_f $ for each $ i \in A $, then $\sup_{i\in A}\mu_{i} \in \tau_f $.
\end{enumerate}
The pair $ (X,\tau_f) $ is called a fuzzy topological space.
\end{definition}
The elements of $ \tau_f $ are called fuzzy open sets in $X$.
 \begin{definition}\cite{kh}
 A fuzzy topological space $ (X,\tau_f) $ is said to be fuzzy Hausdorff if for $ x,y \in X $ and $ x\neq y $ there exist $ \eta, \beta \in \tau_f $ with $ \eta(x)=\beta(y)=1  $ and $ \eta \wedge \beta=0 $.
 \end{definition}
  A mapping $f$ from a fuzzy topological space $X$ into a fuzzy topological space $Y$ is called fuzzy continuous if $f^{-1}(\mu)$ is fuzzy open in $ X $ for each fuzzy open set $\mu$ in $Y$.

 Suppose $X$ is a fuzzy topological space and $x\in X $. A fuzzy set $\mu$ in $ X $ is called a neighborhood of $x\in X$ if there is a fuzzy open set $\eta$ with $ \eta\leq \mu $ and $ \eta(x)=\mu(x)>0 $. Warren \cite{k11} has proved that a fuzzy set $\mu$ in $X$ is fuzzy open if and only if $\mu$ is a neighborhood of $x$ for each  $ x\in X $ with $\mu(x)>0$. The collection $\mathcal{O}$ of fuzzy sets is called a fuzzy open covering of $\mu$ when $\mu \leq \sup_{\nu\in\mathcal{O}}\nu$. The fuzzy set $\mu$ is called fuzzy compact if any fuzzy open covering of $\mu$ has finite subcovering.

 \begin{definition}\cite{kh}
 If $ \mu_{1}$ and $\mu_{2} $ are two fuzzy subsets of a vector space $ E $, then the fuzzy set $ \mu_1+\mu_2 $ is defined by
 $$ (\mu_1+\mu_2)(x)=\sup_{x=x_1+x_2} (\mu_1(x_1)\wedge \mu_2(x_2)). $$
  If $ t \in \mathbb{K}, $ we define the fuzzy sets $ \mu_1 \times \mu_2  $ and $ t\mu $ as follows:
$$ (\mu_1 \times \mu_2)(x_1,x_2)= \min \lbrace \mu_1(x_1), \mu_2(x_2)\rbrace $$
and
\begin{enumerate}
\item[$(i)$] \;\; for $ t\neq 0,\: (t\mu)(x)=\mu(\frac{x}{t}) $ \: for all \: $ x \in E $,
\item[$(ii)$]\:\: for $ t=0, \: (t\mu)(x)= \begin{cases}0 &   x\neq 0, \quad \\ \sup_{y \in E} \mu(y) &x=0. \end{cases} $
\end{enumerate}
 \end{definition}
%  A net $(x_\alpha)_{\alpha\in T}$ in a fuzzy topological vector space $E$ is fuzzy convergent to $x$ if and only if for each fuzzy neighborhood $\mu$ of $x$ and each $0<\varepsilon<\mu(x)$ there is $\alpha_o\in T$ such that $\mu(x_\alpha)>\varepsilon$ for all $\alpha>\alpha_0$. A net $(x_\alpha)_{\alpha\in T}$ in a fuzzy topological vector space $E$ is fuzzy Cauchy if and only if for each fuzzy neighborhood $\mu$ of zero and each $0<r<\mu(0)$ there is $\alpha_0\in T$ such that for every $\alpha,\alpha'\in T$ with $\alpha,\alpha'>\alpha_0$, $|\mu(x_\alpha-x_{\alpha'})|>r$. A fuzzy topological vector space $E$ is called complete if and only if each Cauchy net in $E$ is convergent \cite{kh}.

 A fuzzy set $ \mu $ in vector space $ E $ is called balanced if $ t\mu \leq \mu $ for each scalar $ t $ with $ \vert t \vert \leq1$. As is shown in \cite{k10}, $ \mu $ is balanced if and only if $ \mu(tx)\geq \mu(x) $ for each $ x\in E $ and each scalar $ t $ with $ \vert t \vert\leq1 $. Also, when $ \mu $ is balanced, we have $ \mu(0) \geq \mu(x) $ for each $ x\in E $. The fuzzy set $ \mu $ is called absorbing if and only if $ \sup_{t>0}t\mu=1 $. Then a fuzzy set $\mu$ is absorbing whenever $\mu(0)=1$. We shall say that the fuzzy set $\mu$ is convex if and only if for all $t\in I $, $t\mu+(1-t)\mu\leq\mu$, \cite{k7}. Also a fuzzy set $ \mu $ is called absolutely convex if it is balanced and convex.

\begin{definition}\cite{dD1}
 A fuzzy topology $\tau_f$ on a vector space $E$ is said to be a fuzzy linear topology, if the two mappings\\
$$ f:E \times E \rightarrow E, (x,y)\rightarrow x+y,$$
$$ g:\mathbb{K} \times E \rightarrow E, (t,x)\rightarrow tx,$$
are continuous when $ \mathbb{K} $ is equipped with the fuzzy topology induced by the usual topology, $ E \times E $ and $ \mathbb{K} \times E $ are the corresponding product fuzzy topologies. A vector space $ E $ with a fuzzy linear topology $ \tau_f $, denoted by the pair $ (E,\tau_f) $ is called fuzzy topological vector space (abbreviated to FTVS).
\end{definition}
\begin{definition}\cite{k7}
Let $(E,\tau_f)$ be a fuzzy topological vector space, the collection $ \nu\subset\tau_f $ of neighborhoods of zero is a local base whenever for each neighborhood $ \mu $ of zero and each $\theta\in(0,\mu(0))$ there is $ \gamma\in \nu $ such that $ \gamma\leq\mu $ and $ \gamma(0)>\theta $.
\end{definition}
\begin{definition}\cite{k8}
A fuzzy seminorm on $E$ is a fuzzy set $\mu$ in $ E $ which is absolutely convex and absorbing.
\end{definition}
We say that a fuzzy set $\mu$, in a vector space $E$, absorbs a fuzzy set $\eta$ if $\mu(0) > 0$ and for every $\theta <\mu(0)$ there exists $t > 0$ such that $\theta\wedge (t\eta) \leq \mu$. A fuzzy set $\mu$ in a fuzzy linear space $E$ is called bounded if it is absorbed by every neighborhood of zero.
%A locally convex fuzzy linear space $E$ is called bornological if every absolutely convex fuzzy set in $E$ which absorbs bounded sets is a neighborhood of zero or equivalently every fuzzy bounded linear operator from $E$ into any fuzzy topological vector space is fuzzy continuous \cite{k8}.
Let $(E,\tau)$ be a topological vector space. Then the collection of all lower semicontinuous functions from $E$ in to $[0,1]$ is a fuzzy linear topology on $E$ denoted by $w(\tau)$.
\begin{proposition}\label{as}
Let $(E,\tau_f)$ be a fuzzy topological vector space. Then there is a linear topology $\tau$ on $E$ such that  $\tau_f=w(\tau)$.
\end{proposition}
\begin{proof}
Let $F=\{\mu^{-1}(r,1]: r\in (0,1]\:\: \textit{and}\:\: \mu\in \tau_f\}$. Now suppose $\tau$ is the linear topology on $E$ which is generated by $F$. It is clear that each $\mu\in \tau_f$ is lower semicontinuous on $(E,\tau)$. Therefore, we have $\tau_f=w(\tau)$.
\end{proof}
In \cite{ro}, the concept of dual pair defined as follows:\\
We call $(E,E')$ a dual pair, whenever $E$ and $E'$ are two vector spaces over the same $ \mathbb{K} $ scaler field and $ \langle x, x' \rangle $ is a bilinear form on $E$ and $ E' $ satisfying the following conditions:
\begin{enumerate}
\item[$(D)$] For each $ x\neq0 $ in $ E $, there is $ x^{\prime}\in E^{\prime} $ such that $ \langle x, x^{\prime}\rangle\neq0 $.\\
\item[$(D^{\prime})$] For each $ x^{\prime}\neq0 $ in $ E^{\prime} $ there is $ x\in E $ such that $ \langle x, x^{\prime}\rangle\neq0 $.\\
\end{enumerate}
 Let $(E,E')$ be a dual pair. We denote by $\sigma_f(E,E')$, the weakest linear fuzzy topology on $E$ which make all $x'\in E'$ fuzzy continuous as linear functionals from $E$ into $(\mathbb{K}, w(\tau))$, where $\tau$ is the usual topology on $\mathbb{K}$. We call $\sigma_f(E,E')$, the weak fuzzy topology.
\begin{lemma} \label{lem12}
Let $(E,E')$ be a dual pair. If we consider the usual fuzzy topology $ \omega(\tau) $ on $\mathbb{K}$, then
 $$\sigma_f(E,E')=\omega(\sigma(E,E')),$$
 where $\sigma(E,E')$ usual weak topology on $E$.
\end{lemma}
\begin{proof}
It is enough to show that for $ f_{1}, f_{2},\cdots, f_{n}\in E' $ and $ \mu_{1}, \mu_{2},\cdots , \mu_{n}\in \omega(\tau) $ the fuzzy set $ \bigwedge^{n}_{i=1}f^{-1}_{i}(\mu_{i}) $ is lower semi-continuous on $ E $ with respect to the topology $ \sigma(E,E'). $ Since $ \mu_{i}\in \omega(\tau) $ for $ i=1,2,\cdots, n $, then $ \mu_{i}: \mathbb{K}\rightarrow I $ are lower semi-continuous. Also $$ f_{i}: (E,\sigma(E,E'))\rightarrow \mathbb{K} $$ is continuous and then is lower semi-continuous. Then $\mu_{i} \circ f_{i} $ is lower semi-continuous for each $ i=1,2,\cdots,n. $ Then $$ \bigwedge^{n}_{i=1}f^{-1}_{i}(\mu_{i})= \bigwedge^{n}_{i=1}\mu_{i} \circ f_{i} $$ is lower semi-continuous on $ (E,\sigma(E,E')) $ for each $ i=1,2,\cdots,n. $
\end{proof}
Let $E$ be a vector space and $ \mu \in I^{E}$. The $\theta$-level set of $\mu$ defined by
$$[\mu]_{\theta}=\{x\in E:\mu(x)\geq\theta\},$$
 where $0<\theta\leq 1$.  It is clear that for $\theta_1,\theta_2\in (0,1]$ with $\theta_1<\theta_2$ we have $[\mu]_{\theta_1}\supseteq [\mu]_{\theta_2}$. Therefore $\{[\mu]_{\theta}:\theta\in(0,1]\}$ is a decreasing collection of subsets of $E$.

\begin{proposition}\label{1a}
Let $(E,\tau_f)$ be a fuzzy topological vector space and $\tau_f=w(\tau)$. Then the followings hold.
\begin{enumerate}
\item[$(a)$] the fuzzy set $\mu\in I^{E}$ is fuzzy compact if and only if and if $[\mu]_{\theta}$ is compact in $(E,\tau)$ for each $0<\theta\leq 1$,
\item[$(b)$] the fuzzy set $\mu\in I^{E}$ is fuzzy closed if and only if and if $[\mu]_{\theta}$ is closed in $(E,\tau)$ for each $0<\theta\leq 1$,
\item[$(c)$] the fuzzy set $\mu\in I^{E}$ is fuzzy absolutely convex if and only if and if $[\mu]_{\theta}$ is absolutely convex in $(E,\tau)$ for each $0<\theta\leq 1$.
\end{enumerate}
\end{proposition}
\begin{proof}
For (a), let $\mu\in I^{E}$ is fuzzy compact and $0<\theta\leq 1$. Suppose $\{A_j\}_{\alpha\in j}$ be an open covering of $[\mu]_{\theta}$. Without loss of generality, we can suppose $A_j=\nu_j^{-1}(r_j,1]$, were $\nu_j\in\tau_f$ and $r_j\in(0,1]$. Therefore if $\mu(x)\geq\theta$ then $\nu_j^{-1}(x)>r_j$. Then $\mu\leq \sup_{j\in T}\nu_j$. Now since $\mu\in I^{E}$ is fuzzy compact, there is $j_1,\cdots,j_n\in T$ such that $\mu< \sup_{i\in \{1,\cdots,n\}}\nu_{j_i}$. This shows that $[\mu]_{\theta}\subseteq \bigcup_{i\in \{1,\cdots,n\}}A_{j_i}$. Then $[\mu]_{\theta}$ is compact in $\tau$. Similarly, the converse holds.

For (b), let $\mu\in I^{E}$ is fuzzy closed. Then $1-\mu$ is fuzzy open. This shows that $(1-\mu)^{-1}(\theta,1]$ is open in $\tau$ for each $0<\theta\leq 1$. Thus $E\setminus((1-\mu)^{-1}(\theta,1])=[\mu]_{1-\theta}$ is closed in $\tau$. Since $\theta$ is arbitrary, then $[\mu]_{\theta}$ is closed. The converse is similar.

The assertion (c) is clear.
\end{proof}

%%%%%%%%%%%%%%%%%%%%%%%%%%%%%%%%%%%%%%%%%%%%%%%%%%%%%%%%%%%%%%%%%%%%%%%%%%%%%
%%%%%%%%%%%%%%%%%%%%%%%%%%%%%%%%%%%%%%%%%%%%%%%%%%%%%%%%%%%%%%%%%%%%%%%%%%%%%
\section{Polar fuzzy sets}

\noindent
In this section, we introduce the concept of polar for fuzzy sets using $\theta$-level sets and investigate their properties.

Let $ (E,E') $ be a dual pair. For the non-empty subset $ A $ of $ E $ its polar $ A^{\circ} $ defined by:
$$ A^{\circ}=\lbrace x' \in E': \:\: \sup_{x \in A}|\langle x, x^{\prime}\rangle| \leq1 \rbrace. $$

\begin{definition}
Let $(E,E')$ be a dual pair, $ \mu \in I^{E} $. For $x'\in E'$ and $ \mu \in I^{E} $, we set $A_{\mu,x'}=\{\theta\in (0,1]:x'\in [\mu]_{1-\theta}^\circ\}$. For $\mu\neq 0$, we define the fuzzy set $\mu^{\circ}$ on $E'$ as follows:
$$\mu^{\circ}(x')=\begin{cases}\sup A_{\mu,x'} \quad ,&A_{\mu,x'}\neq\emptyset,\\0,&A_{\mu,x'}=\emptyset,\end{cases}$$
and call it the fuzzy polar of $\mu $ in $ E$.
\end{definition}
\begin{remark}
We note that if $\mu=1$, then for $\theta\in(0,1]$ we have $[\mu]_{1-\theta}=E$. Therefore $[\mu]_{1-\theta}^\circ=E^\circ=\{0\}$. This shows that for $x'\neq 0$, $A_{\mu,x'}=\emptyset$. Then $\mu^{\circ}(x')=0$ but clearly we have $\mu^{\circ}(0)=1$.
\end{remark}
\begin{lemma}
Let $(E,E')$ be a dual pair and $ \mu \in I^{E} $. Then we have
$$\mu^\circ=\sup_{\theta\in(0,1]} \theta\wedge \chi_{_{[\mu]_{1-\theta}^\circ}}.$$
\end{lemma}
\begin{proof}
The proof is clear.
\end{proof}
\begin{proposition}
Let $ (E,E') $ be a dual pair and $ \mu \in I^{E} $. Then $[\mu^\circ]_{\theta}=[\mu]_{1-\theta}^\circ$.
\end{proposition}
\begin{proof}
Suppose that $x'\in[\mu^\circ]_{\theta}$ for $\theta\in(0,1]$. Then $\mu^\circ(x')\geq\theta$. This shows that $\sup \{\alpha: x'\in[\mu]_{1-\alpha}^\circ\}\geq\theta$. Then there is $\alpha\in(0,1]$ such that  $x'\in[\mu]_{1-\alpha}^\circ$ and $\alpha\geq\theta$. Since $1-\alpha\leq 1-\theta$, we have $[\mu]_{1-\theta}\subseteq [\mu]_{1-\alpha}$ and then $[\mu]_{1-\alpha}^\circ\subseteq [\mu]_{1-\theta}^\circ$. This shows that $x'\in [\mu]_{1-\theta}^\circ$. Similarly, one can prove that $[\mu]_{1-\theta}^\circ\subseteq[\mu^\circ]_{\theta}$. Therefore $[\mu^\circ]_{\theta}=[\mu]_{1-\theta}^\circ$.
\end{proof}
\begin{proposition}
Let $ (E,E') $ be a dual pair. Then the fuzzy polars in $ E'$ have the following property.
\begin{enumerate}
\item[$(a)$] If $ \mu \leq \rho $, then $ \rho^{\circ} \geq \mu^{\circ} $.
\item[$(b)$] If $ \lambda \neq 0 $, then $(\lambda \mu)^{\circ}=\dfrac{1}{|\lambda|}\mu^{\circ}$.
\item[$(c)$] $ (\bigvee_{j \in T}\mu_j)^{\circ}=\bigwedge_{j \in T}(\mu_j)^{\circ} $.
\item[$(d)$] If $ B\subseteq E $, then $ (\chi_{_{B}})^{\circ}= \chi_{_{B^\circ}} $.
\item[$(e)$] $ \mu^{\circ} $ is absolutely convex.
\end{enumerate}
\end{proposition}
\begin{proof}
%For (a) we note the for each $ x\in E; \varphi_x(x')=<x,x'> $ is a $ \sigma_f(E',E)- $continuous linear functional on $ E'. $ Now, since $ (\theta \wedge x_E) $ where $ E=[-1, 1] $ is a closed subset of $ \mathbb{K} $ and $ \mu_{\theta}^{\circ}=\bigwedge_{\mu(x)>\theta}\varphi_x^{-1}(\theta \wedge x_E) $ then $ \mu_{\theta}^{\circ} $ is $ \sigma_f(E',E)- $closed.\\
 For (a), suppose $ \mu \leq \rho. $ Then for $ x\in E $ with $ \mu(x)\geq\theta $, we have $ \rho(x)\geq \theta$. This shows that $[\mu]_{\theta}\subseteq [\rho]_{\theta}$ for each $\theta\in(0,1]$. Therefore $[\mu]_{1-\theta}^\circ\supseteq[\rho]_{1-\theta}^\circ$. This implies that $A_{\rho,x'}\subseteq A_{\mu,x'}$ for each $x'\in E'$. Therefor $\mu^\circ(x')=\sup A_{\mu,x'}\leq\sup A_{\rho,x'}=\rho^\circ(x')$.
\\
 For (b), let $ \lambda \neq0$. For $x'\in E'$, suppose  $A_{\lambda\mu,x'}\neq \emptyset$. We have
\begin{align}
  (\lambda\mu)^{\circ}(x')&=\sup A_{\lambda\mu,x'}
  \nonumber\\ &=\sup\{\theta\in(0,1]:x'\in [\lambda\mu]_{1-\theta}^\circ\}\nonumber \\&= \sup\{\theta\in(0,1]:x'\in (\lambda[\mu]_{1-\theta})^\circ\}\nonumber \\&=\sup\{\theta\in(0,1]:x'\in \dfrac{1}{|\lambda|}([\mu]_{1-\theta})^\circ\}\nonumber\\&=
 \sup\{\theta\in(0,1]:|\lambda|x'\in ([\mu]_{1-\theta})^\circ\}\nonumber \\&=\mu^\circ(|\lambda|x')=\dfrac{1}{|\lambda|} \mu^\circ(x')\nonumber.
  \end{align}
  For (c), let for each $ j\in T, \mu_j \in I^{E} $ we have
  \begin{align}
  (\bigvee_{j \in T}\mu_j)^{\circ}(x')&=\sup A_{_{\bigvee_{j \in T}\mu_j,x'}}=\sup \{\theta\in(0,1]: x'\in (\bigcup_{j\in T}[\mu_j]_{1-\theta})^\circ\}
  \nonumber\\ &=\sup \{\theta\in(0,1]: x'\in \bigcap_{j\in T}([\mu_j]_{1-\theta})^\circ\}\nonumber \\&=\bigwedge_{j\in T} (\mu_j)^{\circ}(x')\nonumber.
  \end{align}
  For (d), let $ B\subseteq E $ and $ B\neq\emptyset $. Then we have
  \begin{align}
  (\chi_{_{B}})^{\circ}(x')&=\sup A_{_{\chi_{_{B}},x'}}
  \nonumber\\ &=\sup\{\theta:x'\in[\mu]_{1-\theta}^\circ=B^\circ\}\nonumber \\&=\chi_{_{B^\circ}}. \nonumber
  \end{align}
  For (e), firstly we prove that $ \mu^{\circ} $ is balanced. So it is enough to show that $ \mu^{\circ}(tx')\geq \mu^{\circ}(x')$ for all $ x' \in E $ and $ t \in \mathbb{R} $ with $ |t|\leq 1$. We have
  \begin{align}
  \mu^{\circ}(tx')&=\sup A_{\mu,tx'}
  \nonumber\\ &=\sup \{\theta\in(0,1]:tx'\in[\mu]_{1-\theta}^\circ\}\nonumber \\&= \sup \{\theta\in(0,1]:x'\in(t[\mu]_{1-\theta})^\circ\}\nonumber \\&=\sup \{\theta\in(0,1]:x'\in[t\mu]_{1-\theta}^\circ\}\nonumber\\&=
\sup A_{t\mu,x'}\nonumber \\&=(t\mu)^{\circ}( x')\nonumber\\&=\frac{1}{|t|} \mu^{\circ}(x')\nonumber
  \\&\geq 1\times \mu^{\circ}( x')=\mu^{\circ}( x') \nonumber.
  \end{align}
  The convexity of $ \mu^{\circ} $ is obvious, since $[\mu]_{1-\theta}^\circ$ is a convex set for each $\theta\in(0,1]$.
\end{proof}
\begin{theorem}\label{145}
If $(E, \tau_f)$ is a Hausdorff locally convex fuzzy topological vector space and $\mu$ is a fuzzy neighborhood of origin, then $\mu^\circ$ is $\sigma_f(E',E)$-compact.
\end{theorem}
\begin{proof}
 We have $\sigma_f(E',E)=w(\sigma(E',E))$ by Lemma \ref{lem12}. Since for each $0<\theta\leq 1$, so the subset $[\mu^\circ]_{\theta}=[\mu]_{1-\theta}^\circ$ of $E'$ is $\sigma(E',E)$-compact, Proposition \ref{1a} implies that $\mu^\circ$ is $\sigma_f(E',E)$-compact.
\end{proof}
\begin{proposition}\label{cm}
If  $(E, \tau_f)$ is a Hausdorff locally convex fuzzy topological vector space and $\mu$ is a closed absolutely convex fuzzy set in $E$, then $\mu^{\circ\circ}=\mu$.
\end{proposition}
\begin{proof}
Let $E'$ be the fuzzy dual of $E$. For $x\in E$, we have
$$\mu^{\circ\circ}(x)=\sup\{\theta:x\in[\mu^\circ]_{1-\theta}^\circ\}$$
Let $\tau_f=w(\tau)$. Since $\mu$ is fuzzy closed and absolutely convex, so by Proposition \ref{1a} $[\mu]_\theta $ is closed and absolutely convex with respect to $\tau$. Then, we have
$$[\mu^\circ]_{1-\theta}^\circ=([\mu]_\theta^\circ)^\circ=[\mu]_\theta^{\circ\circ}=[\mu]_\theta.$$
This shows that $$\mu^{\circ\circ}(x)=\sup\{\theta:x\in[\mu^\circ]_{1-\theta}^\circ\}=\sup\{\theta:x\in[\mu]_{\theta}^{\circ\circ}\}=\sup\{\theta:x\in[\mu]_{\theta}\}=\mu(x).$$
\end{proof}
\begin{corollary}\label{cor1}
Let $E$ be a vector space and $\mu\in I^E$. Then we have $\mu\leq\mu^{\circ\circ}$.
\end{corollary}
\begin{proposition}\label{p1}
Let $ (E, \tau_f) $ be a locally convex fuzzy topological vector space. If $ \mathcal{V} $ is neighborhood base for zero in $ E $ then $ E'_f=\bigcup_{\mu \in \nu}E_\mu, $ where $ E_\mu=\lbrace x \in E^* ;\:\: \mu^\circ(x)>0 \rbrace  $
\end{proposition}
\begin{proof}
Let $ x' \in E^\ast $. The linear functional $ x' $ is fuzzy continuous on $ E $ if there is $ \mu \in \mathcal{V} $ such that
$$ x'(\mu) \leq \chi_{_{H}},\:\: H=[-1, 1]. $$ Then $ x'(\mu)(z) \leq \chi_{_{H}}(z) $ for each $ z \in \mathbb{R}. $ This shows that
$$ \sup_{x'(y)=z}\mu(y) \leq \chi_{_{H}}(z). $$
Or equivalently, for each $ y \in E $ with $ x'(y)=z, $
\begin{align}
  \mu(y)&\leq \chi_{_{H}}(z)
  \nonumber\\ &=\chi_{_{H}}(x'(y))
  \nonumber\\ &=\chi_{_{H}}(\langle y, x^{\prime}\rangle).\nonumber
  \end{align}
 If $\mu(x)\geq 1-\theta$ for some $\theta\in(0,1]$, then $\chi_{_{H}}(\langle x, x^{\prime}\rangle)\geq 1-\theta$. Then $\chi_{_{H}}(\langle x, x^{\prime}\rangle)=1$ or $|\langle x, x^{\prime}\rangle|\leq 1$. This shows that $x'\in[\mu]_{1-\theta}^\circ $.
  Then $ \mu^{\circ}(x')\geq\theta>0 $ and $ x' \in E_\mu. $
\end{proof}
%\begin{definition}
%A fuzzy set $ \mu $ in a fuzzy linear space $ E $ is called bounded if it is absorbed by every fuzzy neighborhood of zero.
%\end{definition}
Let $ (E,\tau_f) $ be a fuzzy topological vector space. The fuzzy closure of the fuzzy set $ \mu $ is the smallest closed fuzzy set which contains $\mu$. Also, the convex envelope and absolutely convex envelope of $\mu$ are the smallest convex and absolutely convex fuzzy sets which contains $\mu$, respectively.

We denote the fuzzy closure, convex envelope and absolutely convex envelope of the fuzzy set $\mu$, by $\overline{\mu}^{f}$, $ co(\mu)$ and $aco(\mu)$, respectively.
\begin{lemma}
The fuzzy closure, convex envelope and absolutely convex envelope of a bounded fuzzy set is bounded.
\end{lemma}
\begin{proof}
Let $ (E,\tau_f) $ be a FTVS. Let $ \mathcal{V} $ be a base of fuzzy closed absolutely convex neighborhoods of zero. If $ \mu \in I^E $ is bounded and $ \varphi \in \mathcal{V}, $ then for $ \theta \in (0, \varphi(0)) $ there is $ t>0 $ such that $ \theta \wedge (t\mu)\leq \varphi$. Since $\varphi $ is fuzzy closed, convex and absolutely convex, so we have
$$ \theta \wedge (t\overline{\mu}^{f})\leq \overline{\varphi}^{f}=\varphi, $$
$$ \theta \wedge (t(co(\mu))) \leq co(\varphi)=\varphi, $$
and
$$ \theta \wedge (t(aco(\mu)) \leq aco(\varphi)=\varphi.$$

\end{proof}
\begin{proposition}
The image of a bounded fuzzy set under a fuzzy continuous linear operator is bounded.
\end{proposition}
\begin{proof}
Let $ (E, \tau_f) $ and $ (F,\tau'_f) $ be fuzzy topological vector spaces and $ T: E \rightarrow F $ be a fuzzy continuous linear operator. Then for each fuzzy neighborhood $ \varphi \in \tau'_f $ of zero, there is fuzzy neighborhood $ \mu \in \tau_f $ such that $ T(\mu) \leq \varphi. $ Now, let $ \eta $ be a fuzzy bounded set in $ E $. Then for each $ \theta \in (0, \mu(0)) $, there is $ t>0 $ such that $ \theta \wedge (t\eta) \leq \mu. $ This shows that
$$ \theta \wedge (tT(\eta)) \leq T(\mu) \leq \varphi. $$
Therefor $ T(\eta) $ is bounded in $ (F,\tau'_f). $
\end{proof}
\begin{theorem}
If a locally convex fuzzy topological vector space has a bounded zero neighborhood, then it is Katsaras seminormable.
\end{theorem}
\begin{proof}
Let $(E,\tau_f)$ be locally convex fuzzy topological vector space and $ \vartheta $ be a fuzzy bounded neighborhood of zero. We set $ \mu=aco(\vartheta)$. Then $\mu$ is absolutely convex and bounded. Since $ \mu $ is fuzzy bounded, for each neighborhood $S$ of zero and each $ \theta \in (0,\mu(0)) $ there is $t>0$ such that
$$\theta\wedge t\mu \leq S.$$
This shows that $ \mathcal{F}=\lbrace \theta\wedge t\mu:t>0, 0\leq \theta<1 \rbrace $ is a neighborhood base at zero. Then $(E,\tau_f)$ is a Katsaras seminormable.
\end{proof}
%\begin{theorem}
%Let $ (E, E') $ be a dual pair. Then $ \mu \in I^E $ is weakly bounded if and only if $ \mu_{\theta}^\circ $ is absorbent for each $ 0<\theta \leq1. $
%\end{theorem}

 %%%%%%%%%%%%%%%%%%%%%%%%%%%%%%%%%%%%%%%%%%%%%%%%%%%%%%%%%%%%%%%%%%%%%%%%%%%%%%

%%%%%%%%%%%%%%%%%%%%%%%%%%%%%%%%%%%%%%%%%%%%%%%%%%%%%%%%%%%%%%%%%%%%%%%%%%%%%
\section{Polar fuzzy topologies}
Let $ (E, E') $ be a dual pair and $ \mathcal{B} $ be a collection of weakly fuzzy bounded subsets of $ E $ satisfying the following conditions:
\begin{enumerate}
\item[$(c_1)$] If $ \mu \in \mathcal{B} $ and $ \varphi \in \mathcal{B} $, then there is $ S \in \mathcal{B} $ such that $ \mu \vee \varphi < S $.
\item[$(c_2)$] If $ \mu \in \mathcal{B}$, then $ t\mu \in \mathcal{B} $ for each $ t \in \mathbb{R} $.
\item[$(c_3)$] For each $x \in E$, there is $ \mu \in \mathcal{B} $ such that $ \mu(x)>0 $.
\end{enumerate}
  Then the polars of the elements of $ \mathcal{B} $ i.e. $ \mathcal{B^{\circ}}= \lbrace \mu^{\circ}:  \mu \in \mathcal{B} \rbrace $ form a neighborhood base at zero for a linear fuzzy topology on $ E' $, namely the polar fuzzy topology.
  Let $ (E,E') $ be a dual pair. If $\mathcal{B}  $ is the collection of all weakly bounded fuzzy subsets of $ E $, then the corresponding polar fuzzy topology is the finest polar fuzzy topology on $ E' $ which is denoted by $ \beta_f(E,E'). $
  \begin{example}
  Let $ (E,E') $ be a dual pair. Also, suppose $ \mathcal{B} $ is the collection of all the fuzzy sets $ A_{\lambda} $ defined as follows: for the finite subset $ A\subset E, $ and $ 0<\lambda \leq 1 $ we set
  $$ A_{\lambda}:E\rightarrow I $$
  $$A_{\lambda}(x)=\begin{cases}\lambda &x \in A,\\0 &else.\end{cases}$$
  Then the collection $ \mathcal{B} $ satisfies the conditions $ (c_1),(c_2) $ and $ (c_3) $. For $ (c_1), $ let $ A_{\lambda},B_{\gamma} \in \mathcal{B}, $ we set $ C=A\cup B $ and $ \alpha= \max \lbrace \lambda,\gamma \rbrace. $ Then we have $ A_{\lambda}\vee B_{\gamma} \leq C_{\alpha}. $ For $ (c_2) $, let $ A_{\lambda} \in \mathcal{B} $ and $ t \in \mathbb{R}. $ We set $ B=tA. $ Now, we have
   \begin{align}
  (tA_{\lambda})(x)&= A_{\lambda}(\frac{x}{t})
  \nonumber\\ &=\begin{cases}\lambda &\frac{x}{t} \in A \\0 &else \end{cases}\nonumber \\&= \begin{cases}\lambda &x \in tA \\0 &else \end{cases}\nonumber \\&=\begin{cases}\lambda &x \in B \\0 &else \end{cases}\nonumber\\&=
  B_{\lambda}(x)\nonumber.
  \end{align}
  For $ (c_3) $, let $ x \in E $. Then $ \chi_{_{A}} \in B $ where $ A=\lbrace x \rbrace $ and we have $ \chi_{_{A}}(x)=1. $ Hence the collection $ \mathcal{B^{\circ}}= \lbrace \mu^{\circ}:  \mu \in \mathcal{B} \rbrace $ is a base at zero for a linear fuzzy topology. We have
 \begin{align}
  (A_{\lambda})^{\circ}(x')&=\sup A_{A_\lambda,x'}=\lambda\wedge \chi_{_{A^\circ}}(x').
  \nonumber
  \end{align}
  So the polar fuzzy topology is generated by fuzzy seminorms $ \chi_{_{A^{\circ}}} $ in this case.
  \end{example}
  \begin{theorem}
  Let $ (E,E') $ be dual pair. Then the following statements hold:
  \begin{enumerate}
\item[$(a)$] any polar fuzzy topology is finer than the weak fuzzy topology $ \sigma_f(E',E) $,
\item[$(b)$] any polar fuzzy topology is fuzzy Hausdorff.
\end{enumerate}
  \end{theorem}
  \begin{proof}
  Let $\mathcal{B}$ be a collection of fuzzy sets of $E$ which have the properties $(a)$, $(b)$ and $(c)$ and $\tau_p$ be the polar fuzzy topology on $E'$. Firstly, we prove that each $x\in E $ is continuous on $E' $. The condition $ (c_3) $ shows that there is $ \mu\in \mathcal{B} $ such that $\mu(x)>0$. Therefore $ \mu^{\circ\circ}(x)\geq \mu(x)> 0$ by Corollary \ref{cor1}. Now, Proposition \ref{p1} shows that $ x \in (E')'$. Then each $x\in E $ is fuzzy continuous on $ E'$. Now, since $ \sigma_f(E',E) $ is the weakest fuzzy topology on $ E' $ which make each $ x \in E $ continuous, we have
  $$ \sigma_f(E',E) \subset \tau_p. $$ Also, since $ \sigma_f(E',E) $ is fuzzy Hausdorff, then $ \tau_p $ is fuzzy Hausdorff.
\end{proof}
\begin{theorem}[Mackey-Arens type Theorem]\label{Ma}
  Let $ (E,E') $ be dual pair and under the linear fuzzy topology $\tau_f$ E be a fuzzy Hausdorff locally convex space. Then $E$ has dual $E'$ if and only if $\tau_f$ is the polar fuzzy topology on a collection of $\sigma_f(E',E)$-compact  fuzzy subsets of $E'$.
\end{theorem}
\begin{proof}
Let $E$ has dual $E'$ under the linear fuzzy topology $\tau_f$ and $\mathcal{V}$ be a base of fuzzy closed neighborhoods of zero. For each $\nu\in\mathcal{V}$, we have $\nu^{\circ\circ}=\nu$ by Proposition \ref{cm}. This shows that $\tau_f$ is the polar fuzzy topology on the collection $\{\nu^\circ:\nu\in\mathcal{V}\}$. Now by Theorem \ref{145}, each $\nu^\circ$ is $\sigma_f(E',E)$-compact. Then $\tau_f$ is the polar fuzzy topology on a collection of fuzzy $\sigma_f(E',E)$-compact subsets of $E'$.

Now, let $\tau_f$ be the polar topology on a collection $\mathcal{K}$ of $\sigma_f(E',E)$-compact fuzzy sets in $E'$,  and $x'$ be a fuzzy continuous linear operator on $E$. By Proposition \ref{p1}, we have $\kappa^{\circ\circ}(x')>0$ for some $\kappa\in\mathcal{K}$.
But we have $\kappa^{\circ\circ}=\kappa$. Then $\kappa(x')>0$. This shows that $x'\in E'$
\end{proof}
\begin{corollary}
Theorem \ref{Ma} shows that there is the finest dual pair linear fuzzy topology on $ E. $ This fuzzy topology is the polar fuzzy topology on the collection of all $ \sigma_f(E,E')$- compact subsets of $ E $. We denote this fuzzy topology by $ \tau_f(E,E'). $
\end{corollary}
\begin{corollary}
For a dual pair $(E,E')$, $ \tau_f(E,E')\subset \beta_f(E,E')$.
\end{corollary}
%%%%%%%%%%%%%%%%%%%%%%%%%%%%%%%%%%%%%%%%%%%%%%%%%%%%%%%%%%%%%%%%%%%%%%%%%%%%%
%%%%%%%%%%%%%%%%%%%%%%%%%%%%%%%%%%%%%%%%%%%%%%%%%%%%%%%%%%%%%%%%%%%%%%%%%%%%%
\section{Conclusion}
The Mackey-Arens Theorem is one of the important theorems in functional analysis. This theorem states that any topology of dual pair can be written as a polar topology and then there is the finest dual pair topology. Theorem \ref{Ma} extends this assertion to fuzzy structures and shows that we can obtain the finest fuzzy dual pair topology.
%%%%%%%%%%%%%%%%%%%%%%%%%%%%%%%%%%%%%%%%%%%%%%%%%%%%%%%%%%%%%%%%%%%%%%%%%%%%%

%

\date{\scriptsize $^{a}$
%Department of Mathematics, University of Maragheh, P. O. Box 55136-553, Maragheh, Iran.\\
E-mail:bdaraby@maragheh.ac.ir,}\\
\date{\scriptsize $^{b}$
%Department of Mathematics, University of Maragheh, P. O. Box 55136-553, Maragheh, Iran.\\
E-mail:nasibeh$_{-}$khosravi@yahoo.com,}\\
\date{\scriptsize $^{c}$
%Department of Mathematics, University of Maragheh, P. O. Box 55136-553, Maragheh, Iran.\\
E-mail:rahimi@maragheh.ac.ir\\

}

\end{document}